\documentclass[11pt]{amsart}
\usepackage[utf8]{inputenc}
\usepackage[english]{babel}
\usepackage{amsmath}
\usepackage{amsfonts}
\usepackage{amssymb}
\usepackage{graphicx}
\usepackage[margin=1.5in]{geometry}

\usepackage{mathrsfs}
\newtheorem{thm}{Theorem}
\newtheorem{res}{Conjecture}
\newtheorem{question}[res]{Open question}
\newtheorem{proposition}[thm]{Proposition}
\newtheorem{corollary}[thm]{Corollary}

\newcommand{\eps}{{\varepsilon}}
\def\sgn{{\rm sgn}}

\newtheorem{conj}[res]{Conjecture}

\newcommand{\R}{\mathbb{R}}

\newcommand{\Z}{\mathbb{Z}}

\newcommand{\dd}{\delta}

\newcommand{\E}{\mathcal{E}}

\author{Wendelin Werner}
\address{
Department of Mathematics,
ETH Z\"urich,
 R\"amistr. 101,
8092 Z\"urich, Switzerland}
\email
{wendelin.werner@math.ethz.ch}
\title {On clusters of Brownian loops in $d$ dimensions}

\begin{document}

\begin{abstract}
We discuss random geometric structures obtained by percolation of Brownian loops,
in relation to the Gaussian Free Field,  and how their existence and properties depend on the dimension of the ambient space. 
We formulate a number of conjectures for the cases $d=3,4,5$ and prove some results when $d > 6$. 
\end {abstract}

\maketitle 

\begin {center} 
{{\it This paper is dedicated to the memory of Vladas Sidoravicius\footnote{The content of this paper corresponds to the last of my talks that Vladas attended
in 2017 and 2018. Like so many of us in the mathematical community, I remember and miss his enthusiasm as well as his contagious, warm and charming smile.}.}}
\end {center} 

\bigbreak

\tableofcontents

\section{Introduction}

Field theory has been remarkably successful in describing features of many models of statistical physics at their 
critical points. In that approach, the focus is put on correlation functions between the values taken by the field at a certain number of given points in space. 
In many instances, these functions correspond to experimentally measurable macroscopic quantities (such as for instance the global magnetization in the Ising model).

Some of these correlation functions can also be directly related to features of conjectural (and sometimes physically relevant) random fractal geometric objects; for instance,  a $2$-point function $F(x_1, x_2)$ can describe 
the asymptotic behaviour as $\eps \to 0$ of the probability that $x_1$ and $x_2$ are both in the $\eps$-neighbourhood of some ``random cluster'' in a statistical physics model 
 -- and the critical exponent that describes the behaviour of $F$ as $y  \to x$ 
is then related to the fractal dimension of the scaling limits of those clusters. 
This type of more concrete geometric interpretation is however not instrumental in the field-theoretical set-up (and for some fields, there is actually no underlying geometric object). 
It remained for a long time rather hopeless to go beyond this aforementioned partial description of these geometric structures via 
correlation functions, due to the lack of other available mathematical tools to define such random geometric objects in the continuum.

In the  very special case of two-dimensions (which is related to Conformal Field Theory (CFT) on the 
field theory side), this has changed with Oded Schramm's construction of Schramm-Loewner Evolutions (SLE processes) in \cite {Schramm}. These are concrete random curves in the plane defined via some mathematical conformally invariant growth mechanism, and that are conjectured to be relevant for most critical systems in two dimensions. 
The Conformal Loop Ensembles (CLE) that were subsequently introduced in \cite {Sh,SW} are random collection of loops, or equivalently 
random connected fractal sets that are built using variants of SLE, and that are describing the (conjectural) scaling limit of the joint law of all clusters in critical lattice models. 
It should be stressed that all these SLE-based developments are relying on conformal invariance in a crucial manner, so that they are specific to the two-dimensional case. 

In this study of two-dimensional and conformal invariant random structures, the following two random objects have turned out to be very closely related to the SLE and CLE:

- The Gaussian Free Field (GFF): As shown in a series of work by Schramm-Sheffield, Dub\'edat and Miller-Sheffield starting with \cite {SS2,Dub,MS1}, 
this random generalized function essentially turns out to host (in a deterministic way) most SLE-based structures. 
There exists for instance a procedure that allows to deterministically draw a CLE, starting from a sample of a GFF. In particular, the SLE$_4$ and the CLE$_4$ appear 
naturally as generalized level lines of the GFF. Of course, it should be recalled that the GFF is also an elementary and fundamental building block in field theory. 

- The Brownian loop-soups: This object, introduced in \cite {LW}, is a Poissonian cloud of Brownian loops in a domain $D$. If, 
as proposed in \cite {Wcras} and shown in \cite {SW}, one considers {\em clusters of Brownian loops}, and their outer boundaries, one constructs also a CLE$_\kappa$ where $\kappa = \kappa (c)$ varies between $8/3$ and $4$ as the intensity $c$ of the loop-soup varies between $0$ and $1$. This intensity plays the role of the {\em central charge} in  the CFT language. 

There is actually a close relation between these two constructions of CLE$_4$ (via the GFF or via the Brownian loop-soup with intensity $c=1$), see \cite {QW} and the references therein. We will come back to this later, but roughly speaking, starting from a sample of a Brownian loop-soup, one can construct a GFF in such a way that the Brownian loop-soup clusters can be interpreted as ``excursion sets'' of the GFF, a little bit like the excursion intervals away from $0$ of one-dimensional Brownian motion, see \cite {ALS2} and the references therein. 

The starting point of the present paper is the observation that both the Brownian loop-soup and the Gaussian Free Field can be defined in any dimension. This leads naturally to wonder what natural random fractal subsets of $d$-dimensional space for $d \ge 3$ can be built using these special and natural objects. In particular, one can guess that just as in two dimensions,  clusters of Brownian loops (for a loop-soup of intensity $c=1$) will have an interesting geometry, and argue that they should be fundamental structures within a GFF sample.  A first immediate reaction is however to be somewhat cautious or even sceptical. Indeed, Brownian loops in dimensions $4$ and higher are simple loops, and no two loops in a Brownian loop-soup will intersect, 
so that a Brownian loop-soup cluster will a priori consist only of one single isolated simple loop. 
But, as we shall explain in the present paper, things are more subtle, and, if properly defined, it should still be possible to agglomerate these disjoint Brownian loops into interesting clusters when the dimension of the space is 4 and 5. 

The structure of the present paper is the following: We will first review some basic facts about Lupu's coupling of the GFF and loop-soups on cable graphs. After discussing heuristically some general aspects of their scaling limits and reviewing the known results in $d=2$, we will make conjectures about the cases $d=3,4,5$. Then, we will state and derive some results for $d > 6$.

We conclude this introduction with the following remark: It is interesting that this loop-soup approach to the GFF bears many 
similarities with the random walk representations of fields as initiated by Symanzik \cite {Symanzik} and further developed by many papers, including by Simon \cite {Si}, the celebrated work by Brydges, Fr\"ohlich and Spencer \cite {BFS} or Dynkin \cite {Dy}. Their motivation was actually to understand/describe ``interacting fields'' (i.e., beyond the free field!) via their correlation functions; given that the correlation functions of the GFF are all explicit, there was then not much motivation to study it further, while the question of existence and constructions of non-Gaussian fields was (and actually still is) considered to be an important theoretical challenge.

\section {Background: Lupu's coupling on cable-graphs} 

A crucial role will be played here by the cable-graph GFF and the cable-graph loop-soup, that have been introduced by Titus Lupu in \cite {Luputhesis,Lupu}. Let us briefly review their main features in this section, and we refer to those papers for details. 

In this section, we consider ${\mathcal D}$ to be a fixed connected (via nearest-neighbour connections) subset of $\Z^d$ (the case of subsets of $\delta \Z^d$ is then obtained simply by scaling space by a factore $\delta$) on which the discrete Green's function is finite. 
We can for instance take ${\mathcal D}$ to be all (or any connected subset) of $\Z^d$ when $d \ge 3$, or a bounded subset of $\Z^2$. The set $\partial {\mathcal D}$ is the set of points that is at distance exactly $1$ of ${\mathcal D}$.  The Green's function $G(x,y)= G_{{\mathcal D}} (x,y)$ is the expected number of visits of $y$ made by a simple random walk in $\Z^d$ starting from $x$ before exiting ${\mathcal D}$ (if this exit time is finite, otherwise count all visits of $y$).  

The cable graph ${\mathcal D}_c$ associated to ${\mathcal D}$ 
is the set consisting of the union of ${\mathcal D}$ with all  edges (viewed as open intervals of length $1$) that have at least one endpoint in ${\mathcal D}$.
One can define also Brownian motion on the cable graph (that behaves like one-dimensional Brownian motion on the edges and in an isotropic way when it is at a site of ${\mathcal D}$).
One can then also define the Green's function $G_{{\mathcal D}_c}$ for this Brownian motion (this time, the boundary conditions correspond to a killing when it hits $\partial {\mathcal D}$) and note that its values on ${\mathcal D} \times {\mathcal D}$ coincide 
with that of the discrete Green's function $G_{\mathcal D}$ for the discrete random walk. 

One can then on the one hand define the Gaussian Free Field (GFF) on the cable graph $(\phi(x))_{x \in {\mathcal D}}$ as a centred Gaussian process with covariance given by the Green's function 
$G_{{\mathcal D}_c}$ on the cable graph. This is a random continuous function on ${\mathcal D}_c$ that generalizes Brownian motion (or rather Brownian bridges) to the case where the time-line is replaced by the graph ${\mathcal D}_c$. The process $(\phi^2 (x))_{x \in {\mathcal D}_c}$ is then called a {\em squared GFF} on ${\mathcal D}_c$. The connected components of $\{ x \in {\mathcal D}_c, \ \phi (x) \not= 0 \}$ are called the {\em excursion sets} of $\phi$ (or equivalently of $\phi^ 2$). 

On the other hand, one can also define a natural Brownian loop measure on Brownian loops on ${\mathcal D}_c$, and then the Brownian loop-soups which are Poisson point processes with intensity given by a multiple $c$ of this 
loop measure. In all the sequel, we will always work with Brownian loop-soups with intensity equal to $c=1$ (in the normalization that is for instance described in \cite {WP} -- in the Le Jan-Lupu normalization that differs by a factor $2$, this would be the loop-soup with intensity $\alpha=1/2$), which is the one for which one can make the direct relation to the GFF. 
Let us make two comments about this loop-soup ${\mathcal L}$ on the cable-graph:  

(i) When one considers a given point on the cable-graph, it will be almost surely visited by an infinite number of small Brownian loops in the loop-soup. However, it turns out that there almost surely exist
exceptional points in the cable-graph that are visited by no loop in the loop-soup (what follows will actually show that the set ${\mathcal Z}$ of such points has Hausdorff dimension $1/2$).  Another equivalent way to define these sets is to first consider {\em clusters of Brownian loops}: We say that two loops $\gamma$ and $\gamma'$ in a loop-soup belong to the same loop-soup cluster, if one can find a finite chain of loops $\gamma_0= \gamma, \gamma_1, \ldots, \gamma_n=\gamma'$ in ${\mathcal L}$ such that $\gamma_j \cap \gamma_{j-1} \not= \emptyset$ for $j =1, \ldots , n$. Then, loop-soup clusters are exactly the connected components of ${\mathcal D}_c \setminus {\mathcal Z}$. 

(ii) Just in the same way in which the occupation time measure of one-dimensional Brownian motion has a continuous density with respect to Lebesgue measure (the {\em local time} of Brownian motion, see e.g. \cite {RY}), each Brownian loop $\gamma$ will have an occupation time measure with a finite intensity $\ell_\gamma$ on the cable graph, so that for all set $A$, the total time spent by $\gamma$ in $A$ is equal to $\int_A \ell_\gamma (x) dx$ where $dx$ denote the one-dimensional Lebesgue measure on ${\mathcal D}_c$. One can then define the ``cumulative'' occupation time density $\Gamma$ of the loop-soup as $\Gamma := \sum_{\gamma \in {\mathcal L}} \ell_\gamma$. This is a continuous function on the cable-graph, that is equal to $0$ on all points of $\partial {\mathcal D}$. Simple properties of Brownian local time show that ${\mathcal Z} = \{ x \in {\mathcal D}_c, \  \Gamma(x) = 0 \}$.

Lupu's coupling between the cable-graph loop-soup and the GFF can now be stated as follows.  

\begin {proposition}[Le Jan \cite {LJ} and Lupu \cite {Lupu}]
\label {prop0}
Suppose that one starts with a Brownian loop-soup ${\mathcal L}$ on the cable-graph ${\mathcal D}_c$. Then the law of its total occupation time density $\Gamma$ is that of (a constant multiple) of a squared GFF. 
Furthermore, if one then defines the function $U=\sqrt {\Gamma}$ and tosses i.i.d. $\pm$ fair coins $\eps_j$ (one for each excursion set $K_j$ of $\Gamma$), then if we write $\eps (x)= \eps_j$ for $x \in K_j$,  the function
$(\eps(x) U(x))_{x \in {\mathcal D}_c}$ is distributed exactly like (a constant multiple of) a GFF on the cable-graph.  
\end {proposition}
In the sequel, we will always implicitly assume that a GFF $\phi$ on a cable-system is coupled to a loop-soup ${\mathcal L}$ in this way. We can note that the excursion sets of $\phi$ are then exactly the loop-soup clusters of ${\mathcal L}$. 

We see that in this setting, the only contribution to the correlation 
between $\phi(x)$ and $\phi (y)$  comes from the event that $x$ and $y$ are in the same loop-soup cluster (we denote this event by $x \leftrightarrow y$), i.e., one has
$$ E [ \phi(x) \phi (y) ] = E [ \eps (x) \eps (y) \times | \phi (x) | \times | \phi(y) | ] 
= E [ |\phi (x)| \times  |\phi (y)| \times  1_{x \leftrightarrow y} ]
$$
for all $x,y$ in ${\mathcal D}_c$.
In the last expression, all quantities are functions of the loop-soup only (and do not involve the $\eps_j$ coin tosses). Similarly, all higher order correlation functions and moments can be expressed only in terms of the cable-graph loop-soup. 

Conversely, since the law of the GFF is explicit and the correlations between $\eps(x)= \sgn ( \phi (x))$ is given in term of cable-graph loop-soup connection events, one gets 
explicit formulas for those connection probabilities. For instance, Proposition \ref {prop0} immediately shows that for all $x,y$ in ${\mathcal D}_c$,
$$E [ \sgn (\phi(x)) \sgn (\phi(y)) ] = E [ \eps (x) \eps (y)] =  P [ x \leftrightarrow y ], $$ 
from which one readily deduces that:
\begin {corollary}[Part of Proposition 5.2 in \cite {Lupu}]
\label {prop1} For all $x \not= y$ in ${\mathcal D}_x$,  
$$ P [ x \leftrightarrow y ] = \arcsin \frac { G(x,y)}{\sqrt {G(x,x) G(y,y) }}.$$ 
\end {corollary} 
In particular, if one considers the cable-graph loop-soup in $\Z^d$ for $d \ge 3$, we see that 
\begin {equation}
 \label {twopoint}
 P [ 0 \leftrightarrow x ] \sim \frac C {\| x \|^{d-2}} 
\end {equation}
for some constant $C$ as $x \to \infty$.

We can note that in this case of $\Z^d$ for $d \ge 3$, $P [ 0 \leftrightarrow x ]$ and $E [ |\phi(0)| | \phi (x)| 1_{0 \leftrightarrow x}]$ 
are comparable when $x \to \infty$. Loosely speaking, this means that when one conditions on $0 \leftrightarrow x$ (and lets $x \to \infty$), the number of small Brownian loops (say of diameter between $1$ and $A$ for a fixed $A$) that pass through the origin does not blow up (this type of considerations can easily be 
made rigorous -- the conditional law of $|\phi(0)|$ in fact remains tight as $x \to \infty$).

\subsection*{Remark} 
{\sl
Throughout this paper, we will always work with loop-soups defined under the very special intensity $c=1$ that makes its occupation time related to the GFF as described above. Understanding features of the ``percolation phase transition''
when the loop-soup intensity varies is a question that will not be discussed here (see \cite {Lupu3,CS,Duetal} and the references therein for results in this direction). }

\section {The fine-mesh and continuum limit}

When $D$ is a connected subset of $\R^d$, in which the continuum Green's function $G_D(x,y)$ is finite when $x \not= y$ (one can for instance think of $D$ to be the unit disk in $\R^2$, or the whole of $\R^d$ when $d \ge 3$), instead of sampling 
a Brownian loop-soup or a continuum GFF directly in $D$, we will consider a Brownian loop-soup and a GFF defined on the cable-graph of a connected fine-grid approximation of $D$ in $\delta {\Z}^d$. For instance (this slightly convoluted definition is just to avoid issues with ``thin'' boundary pieces), if $z_0$ is a given point in $D$, when $\delta$ is small enough, we can choose $D_{\delta}$ to be the connected component of the set of points in $\delta \Z^d$ that are at distance at least $\delta$ from the complement of $D$, and that contains the points that are at distance less than $\delta$ from $z_0$. One can then consider its cable graph $D_{\delta, c}$ and the corresponding GFF and loop-soups as in the previous section (just scaling space by a factor $\delta$). 

We now discuss what happens in the fine-mesh limit (when $\delta \to 0$). 
To avoid confusion, we will use the following terminology: 

- The {\em cable-graph loop-soup} and the {\em cable-graph clusters} will respectively be the soup of Brownian loops defined on the cable graph $D_{\delta,c}$ and the corresponding collection of clusters. 

- The {\em Brownian loop-soup} will be the usual continuum Brownian loop-soup in $D$. The clusters that are created via intersecting Brownian loops will be called {\em Brownian loop-soup clusters}.

Now, when the mesh of the lattice $\delta$ goes to $0$, one can consider the joint limit in distribution of the cable-graph loop-soup, of the corresponding cable-graph clusters and of the cable-graph GFF, and make the following observations:

(i) [About the limit of the loop-soup]   If one sets any positive macroscopic cut-off $a$, then the law of the loops in the cable-graph loop-soup which have a diameter greater than $a$ does converge to that of the loops with diameter greater than $a$ 
in a (continuum) Brownian loop-soup in $D$. This follows from rather standard approximations of Brownian motion by random walks (see \cite {LTF} for this particular instance). So, in that sense, the scaling limit of the cable graph loop-soup is just the Brownian loop-soup in $D$.  
 By Skorokhod's representation theorem, we can also view a Brownian loop-soup in $D$ as an almost sure limit of cable-graph loop-soups.

(ii) [About the limit of the cable-graph GFF] The cable-graph GFF does converge in law to the continuum GFF, because the correlation functions of the cable-graph GFF converge to those of the continuum GFF (all this is due to elementary consideration on Gaussian processes). It should however be stressed that the continuum GFF is not a random function anymore (see for instance \cite {WP}) so that this weak convergence has to be understood in the appropriate function space. 

(iii) [A warning when $d \ge 4$] While the GFF and the Brownian loop-soup are well-defined in any dimension, it is possible to make sense neither of the (renormalized) square of the GFF nor of the (renormalized) total occupation time measure of the Brownian loop when  $d \ge 4$. This is due to the fact that the total occupation time of the Brownian loops 
of diameter in $[2^{-n}, 2^{-n+1}]$ inside a box of size $1$ will have a second moment of the order of a constant times $2^{n(d-4)}$, which is not summable as soon as $d \ge 4$ (so that the fluctuations of the occupation times of the very small loops will outweigh those of the macroscopic ones). 
Since the relation between the cable-graph GFF and the cable-graph loop-soup did implicitly involve the square of the cable-graph GFF, this indicates that some caution is needed when one tries to tie a direct relation between the continuum GFF and the Brownian loop-soup in $\R^d$ when $d \ge 4$. 

Despite (iii), one can nevertheless always study the joint limit of the coupled cable-graph GFF and  cable-graph loop-soup (and its clusters). 
The correlation functions of the cable-graph GFF do provide information on the structure of the cable graph clusters, and
therefore on their behaviour as $\delta \to 0$, as illustrated by Corollary~\ref {prop1}.
One key point is that the scaling limit of the cable graph clusters (if they exist) might be strictly larger than the Brownian loop-soup clusters. Indeed, cable graph clusters may contain loops of macroscopic size (say, some of 
the finitely many loops of diameter greater than some cut-off value  $a$), but they will also contain many small loops, for instance of diameter comparable to the mesh-size $\delta$, or to  $\delta^b$ for some positive power $b$. 
All these small loops do disappear from the loop-soup in the scaling limit if one uses the procedure described in (i), but (just as critical percolation does create macroscopic clusters made of union of edges of size  equal to the mesh-size, while each individual edge does ``disappear'' in the scaling limit) their 
cumulative effect in terms of contributing to create macroscopic cable graph clusters does not necessarily vanish.

In the fine-mesh limit, there a priori appear to be four possible likely scenarios (for presentation purposes, we will consider in the remaining of this section that $D$ is the hypercube $(0,1)^d$): 

\begin {itemize} 
\item 
Case 0. 
There is no limiting joint law for the cable graph clusters when $\dd \to 0$. This should for instance be the case when the number of macroscopic cable graph clusters in $D_\delta$  tends to infinity as $\dd \to 0$. We will come back to this interesting case later. In the remaining cases 1, 2a and 2b, we will assume that the number of cable graph clusters of diameter greater than any fixed $a$ remains tight, and that their joint law has a scaling limit as $\dd \to 0$. 

 \item 
Case 1: 
The limit of the family of macroscopic cable graph clusters is exactly the family of macroscopic clusters Brownian loop-soup clusters.  This means that in this case, the effect of the microscopic loops disappears as $\delta$ vanishes. 

\item Case 2: The limit of the cable graph clusters consists of macroscopic Brownian loops that are somehow agglomerated together also by the effect of the microscopic loops (i.e., the limit of the cable graph clusters are strictly larger than the clusters of macroscopic Brownian loops). Here, the limit of the cable graph clusters would consist of a combination of macroscopic effects and microscopic effects. There are actually two essentially  different subcases: 

- Case 2a: The glueing procedure does involve additional randomness (i.e., randomness that is not present in the Brownian loop soup).

- Case 2b: The 
glueing procedure of how to agglomerate the macroscopic loops is a deterministic function of these macroscopic loops (i.e.,  the limit of the cable-graph clusters is a deterministic function of the corresponding Brownian loop-soup). 

\end {itemize}

Let us summarize here already the conjectures that we will state more precisely in the next sections. We will conjecture that each of the four cases 0, 1, 2a and 2b do occur for some value of the dimension. More specifically, in dimension $d=2$, it is known that Case 1 holds, and we believe that this should also the case when $d=3$, although a proof of this fact appears to remain surprisingly elusive at this point. So, in those lower dimensions, only the macroscopic (in the scaling limit, Brownian) loops 
prevail to construct the excursion sets of the GFF. 
For intermediate dimensions, microscopic loops will start to play an important role: 
As we will try to explain, it is natural to expect that Case 2b holds for $d=4$ and that Case 2a holds for $d=5$. These are two quite fascinating instances, with an actual interplay between microscopic and macroscopic features.

In higher dimensions, one can adapt some ideas that have been developed in the context of (ordinary) high-dimensional percolation to show that Case 0 holds. There is no excursion decomposition of the continuum GFF anymore, but a number of instructive features can be highlighted. A ``typical'' large cable graph cluster will actually contain no macroscopic Brownian loop (even though some exceptional clusters will contain big Brownian loops). Hence, this loop-soup percolation provides a simple percolation-type model that somehow explains ``why'' general high-dimensional critical percolation models should exhibit ``Gaussian behaviour. Indeed, the collection of all these cable-graph clusters is actually very similar to that of ordinary percolation (as they are constructed using only small loops of vanishingly small size at macroscopic level). This in turn sheds some light onto some of the lace-expansion ideas. 

We will now discuss separately the different dimensions. We will first briefly review what is known and proved when $d=2$ and mention the conjectures for $d=3$. 
We will then heuristically discuss the cases $d=4$ and $d=5$ and make some further conjectures, based on some analogies with features of critical percolation within Conformal Loop Ensembles. 
Finally, we will state and prove some results in the case where the dimension is greater than $6$. We note that we will (as often in these percolation questions) not say anything about the ``critical'' case $d=6$ here.

\section {Low and intermediate dimensions} 

\subsection {Low dimensions} 

\subsubsection {Review of the two-dimensional case} 

This is the case where the behaviour of the scaling limit of cable-graph loop-soup clusters is by now essentially fully understood. Indeed, in this case, one has an additional direct good grip on features of the continuum GFF that are built on its coupling with the SLE$_4$ curves (as initiated in \cite {SS2}) and the CLE$_4$ loop ensembles. The paper \cite {SW} provides an explicit description of the Brownian loop-soup clusters as CLE$_4$ loops, so that one can deduce some explicit formulas (such as in \cite {WW}) for the laws of these clusters. These formulas turn out to match exactly the ones that appear in the scaling limit of cable-graph clusters (in the spirit of the formulas by Le Jan \cite {LJ}), so that one can conclude (this is one of the main results of \cite {Lupu2}) that the scaling limit of the cable-graph loop-soup clusters are exactly the Brownian loop-soup clusters (see also some earlier discussion of this problem without the cable-graph insight in \cite {vdBCL}). 

It is then actually possible to push this further: One important result in \cite {ALS1,ALS2} is  that if one associates to each Brownian loop-soup cluster $C_j$ a particular ``natural'' measure $\mu_j$ supported on $C_j$ (which is a  deterministic function of this cluster $C_j$), then, if $(\eps_j)$ are i.i.d. $\pm 1$ fair coin flips, the sum 
$ \sum_j \eps_j \mu_j $ (viewed as an $L^2$ limit) is actually a continuum GFF. In other words, the Brownian loop-soup clusters provide indeed a loop-soup based ``excursion decomposition'' of the continuum GFF despite the fact that the GFF is not a continuous function (it is only a generalized function).

\subsubsection {Conjectural behaviour in dimension 3}

When $d=3$, one can recall that Brownian paths (and loops) have many double points (the Hausdorff dimension of the set of double points in actually equal to $1$). Hence, a Brownian loop in a Brownian loop-soup will almost surely intersect infinitely many other Brownian loops in this loop-soup. From this, one can actually deduce that the Hausdorff dimension $\Delta$ of the Brownian loop-soup clusters is almost surely greater than 2 
($0-1$ law arguments show that this dimension actually always takes the same constant value). On the other hand, Corollary \ref {prop1} can be used to prove that $\Delta$ can not be larger than $5/2$.  
It is natural to conjecture that:

\begin {conj}
 Just as in two dimensions, the scaling limit of the cable-graph loop-soup clusters in three dimensions should exactly be the collection of Brownian loop-soup clusters. The dimension $\Delta$ of these clusters should be equal to $5/2$.
\end {conj} 

One difficulty in proving this conjecture is to be able to exclude the somewhat absurd-looking scenario that in the limit $\delta \to 0$, there  might exist infinitely many disjoint dense (and ``very skinny'') cable-graph loop-soup clusters.

\subsubsection {A further open question} 

When $d=2$, it is known that the obtained loop-soup clusters are in fact a deterministic 
function of the continuum GFF (based on the fact that their boundaries are level lines of this GFF in the sense of \cite {MS1}), so that this ``excursion decomposition'' of the GFF is indeed unique (see \cite {ALS1,ALS2}). 

Let us also recall that when $d=2$ and $d=3$, it is possible to define the (renormalized) square of the continuum GFF (or equivalently, the renormalized total occupation time measure of the loop-soup), see for instance \cite {QW} and the references therein.
Let us now mention a related open question (also to illustrate that some questions remain also in the two-dimensional case).
\begin {question} 
In dimension $d=2$ and $d=3$: Are the (scaling limits of the) loop-soup clusters a deterministic function of this (renormalized) square of the continuum GFF? If not, what is the missing randomness? 

In dimension $d=3$: Are the (scaling limits of the) loop-soup clusters a deterministic function of the continuum GFF? 
\end {question}

\subsection {Intermediate dimensions} 

\subsubsection {Some a priori estimates} 

Again, the cable-graph loop-soup clusters do not proliferate in the $\delta \to 0$ limit, then it is to be expected, based on estimates such as Corollary~\ref {prop1} that the dimension of the scaling limits would be $\Delta=1 + (d/2)$. In particular, when $d = 4$ and $d=5$, if one adds another independent macroscopic Brownian loop to an existing loop-soup, this additional loop will almost surely intersect infinitely many of these limits of cable-graph clusters. From this, it is easy to deduce that a limit of cable-graph clusters would actually contain infinitely many Brownian loops. 

Recall however that a Brownian loop is almost surely a simple loop and that almost surely, any two loops in the loop-soup will be disjoint, so that Brownian loop-soup clusters will all consist of just one loop each (and therefore have Hausdorff dimension equal to $2$). 

Finally, self-similarity of the construction suggests that 
Brownian loops will be part of the scaling limit of the cable graph loop-soups at every scale, and that if one removes all Brownian loops of size greater than $a$ say, then as $a \to 0$, the size of the largest limiting cluster will also vanish. In other words, the ``macroscopic'' loops are instrumental in the construction of the Brownian loop-soup clusters. 

Let us summarize part of this in terms of a concrete conjecture. 
\begin {conj} 
When $d=4$ and $d=5$, the limit in distribution of the cable-graph clusters in $(0,1)^d \cap \delta \Z^d$ does exist, and it is supported on families of clusters of fractal dimension $1 + (d/2)$ with the property that for all small $a$, the number of clusters of diameter greater than $a$ is finite.     
\end {conj}

The main additional heuristic question that we will now discuss is whether the disjoint Brownian loops in the loop-soup get agglomerated into these scaling limit of cable-graph clusters in a deterministic manner or not (i.e., are the scaling limit of the cable-graph clusters a deterministic function the collection of Brownian loops or not?). 

\subsubsection {Background and analogy with CLE percolation} 
It is worthwhile to draw an analogy with one aspect of the papers \cite {MSW1,MSW2} about the existence of a non-trivial ``critical percolation'' model in a random fractal domain. 
Here, one should forget that CLE$_\kappa$ for $\kappa \in (8/3, 4]$ is related to loop-soups or to the GFF, 
and one should view it as an example of random fractal ``carpet'' in the square $[0,1]^2$. The CLE$_\kappa$ carpet $K_\kappa$ in $[0,1]^2$ is obtained by 
removing from this square a countable collection of simply connected sets, that are all at positive distance from each other. 
It can be therefore be thought as a conformal randomized version of the Sierpinski carpet. The following features are relevant here: 

- The larger $\kappa$ is, the smaller the CLE$_\kappa$ tends to be. It is actually possible (this follows immediately from the CLE construction via loop-soups in \cite {SW}) to couple them in a decreasing way i.e., $K_\kappa \subset K_{\kappa'}$ when $8/3 < \kappa' \le \kappa \le 4$.   

- There is one essential difference between CLE$_\kappa$ for $\kappa < 4$ and CLE$_4$: When $\kappa < 4$, there exists a positive $u(\kappa)$ such that for all $a>0$, the probability that there exists two 
holes in $K_\kappa$ that have diameter greater than $a$ and are at distance less than $\eps$ from each other does decay (at least) as $\eps^{u + o(1)}$ as $\eps \to 0$.  
This property fails to hold for CLE$_4$ (this probability will decay logarithmically) which intuitively means that exceptional bottlenecks are more likely in that CLE$_4$ case.

One of the results of \cite {MSW1} is the construction of a process that can be interpreted as a critical percolation process within the random set $K_\kappa$. One can view this either as defining a collection of clusters that live within $K_\kappa$, or if one looks at 
the dual picture, as a collection of clusters that ``glue'' the different CLE loops together (in the original percolation picture, the loops and their interior are ``closed'' and in the dual one, they are now ``open''). In this dual picture, this does therefore 
construct a natural way to randomly regroup these holes (or their outer boundaries, that are SLE-type loops) into clusters.  

One of the results of \cite {MSW2} is that this percolation/clustering procedure is indeed random (i.e., the obtained clusters are not a deterministic function of the CLE$_\kappa$) as long a $\kappa < 4$. 
On the other hand, it is shown in \cite {MSW1} that no non-trivial clustering mechanism can work for CLE$_4$.

\subsubsection {Conjectures} 
The complement of a Brownian loop-soup in $(0,1)^d$ for $d \ge 4$ has some similarities with the previous CLE$_\kappa$ case. It is the complement of a random collection of disjoint simple loops, with a fractal structure. When $d \ge 5$, the ``space'' in-between the loops is much larger than in the $4$-dimensional case, in the sense that the probability 
that two macroscopic loops are $\eps$-close decays like a power of $\eps$, while it only decays in a logarithmic fashion in $4$ dimensions. 
Further analogies can also be made, that lead to: 

\begin {conj}
When $d = 5$, we conjecture that ``critical percolation'' in the space defined by ``contracting all the loops in a loop-soup'' (or equivalently, percolation that tries to glue together the loops in a loop-soup) should exist and be non-trivial. In other words, by observing the Brownian loop-soup only, one does not know which Brownian loops do belong to the same clusters.

When $d = 4$, we conjecture that the glueing mechanism  is deterministic. In other words, by observing the Brownian loop-soup only, one knows which Brownian loops do belong to the same clusters. 
\end {conj}

Let us finally conclude with the same question as for $d=3$: 
\begin {question} 
When $d=4, 5$: In the scaling limit (taking the joint limit of the cable-graph clusters and of the GFF), are the limits of the cable-graph clusters determined by the limiting GFF? 
\end {question}

\section {High dimensions ($d > 6$)} 

\subsection {General features}
As opposed to the cases $d=3,4,5$ where most features are conjectural, it is possible to derive a number of facts when the dimension of the ambient space becomes large enough (note that we will not discuss the somewhat complex case $d=6$ here). 
As opposed to the lower-dimensional cases, these results do not say anything about geometric structures within the  continuum GFF, but they provide insight into the asymptotic behaviour of the cable-graph loop-soup clusters in $\Z^d$ (or in large boxes in $\Z^d$). Actually, when the dimension of the space is large enough, we expect (see \cite {Werner}) that the Brownian loop-soup in $\R^d$ (appearing as the scaling limit of the cable-graph loop-soup) and the GFF (appearing as the limit of cable-graph GFF constructed using the cable-graph loop-soup) become asymptotically independent. 

It is worth first recalling some of the results about usual (finite-range) critical percolation in high dimensions (see \cite {HS, H, HS2, HS3, Aizenman, HvdHS, FvdH, HvdH} and the references therein). 
A landmark result in the study of those models is that when $d$ is large enough, the ``two-point function'' 
(i.e., the probability that two points $x$ and $y$ belong to the same cluster)
behaves (up to a multiplicative constant) like $1/ \|x-y\|^{d-2}$ as $\| x-y\| \to \infty$.
This is known to hold for (sufficiently) spread-out percolation in $\Z^d$ for $d > 6$, and in the case of usual nearest-neighbour percolation for $d \ge 11$. The existing proofs are based on the lace-expansion techniques (that have also been successfully applied to other models than percolation) as developed in this context by Hara and Slade \cite {HS,H,HS2,HS3}). 
This estimate is then the key to the following subsequent statements that we describe in rather loose terms here (see Aizenman \cite {Aizenman}):
 If one considers a finite-range percolation model restricted to $[-N,N]^d$, for which the two-point function estimates is shown to hold, then as $N \to \infty$: 
 
- Clusters with large diameter (say, greater than $N/2$) will proliferate as $N \to \infty$ -- their number will be greater than  $N^{d-6 + o(1)}$ with high probability. 
 
- With high probability, no cluster will have more than $N^{4 + o(1)}$ points in it. 

Note also that the geometry of large clusters can be related to superbrownian excursions. 

\medbreak
As we shall explain now (and plan to discuss in more detail in \cite {Werner}), similar results hold true for the loop-soup clusters in the cable-graph of $\Z^d$ when $d > 6$. The general feature is that the behaviour of the two-point function in this case is given for free by Corollary \ref {prop1}, so that the difficult lace-expansion ideas are not needed here. 
One just has to adjust ideas such as developed by Aizenman in  \cite {Aizenman} on how to extract further information from the estimate on the two-point function. 

\subsection {Some results}
Let us now explain how to adapt some arguments of \cite {Aizenman} to the case of loop-soup percolation. It is convenient to work in the following setting: 
We define $\Lambda_{N}$ to be the set of integer lattice points in $[-N,N]^d$, and $\Lambda_{N,c}$ the cable graph associated to it. 
We will consider the cable-graph loop-soup on $\Lambda_{N,c}$ and study its clusters and connectivity properties. We denote by $n_0$ the number of cable-graph clusters that contain at least one point of $\Lambda_N$, and we order them using some deterministic rule as $C_1, \ldots, C_{n_0}$. We denote by $|C|$ the number of points of $\Lambda_N$ that lie in a set $C$, and when $x \in \Lambda_{N}$, we call $C(x)$ the cluster that contains $x$.
In the sequel, $x \leftrightarrow y$ will always denote the event that $x$ and $y$ are connected via the cable-graph loop-soup in $\Lambda_{N,c}$ (the dependency on $N$ will always be implicit). 
Note that for all $k \ge 1$,
$$ E [ |  C(x) |^k ] = \sum_{ y_1, \ldots, y_k \in \Lambda_N} P [ x \leftrightarrow y_1, \ldots, x \leftrightarrow y_k ].$$
and also that 
$$  E [\sum_{n \le n_0} |C_n|^{k+1} ] =  \sum_{x \in \Lambda_N} E [ | C(x)|^{k}| ].$$ 
Corollary \ref {prop1} then implies (using simple bounds on the Green's function in a box) immediately that there exist constants 
$v_1, v_2$ such  that for all sufficiently large $N$, 
$$ v_1 N^{2} \le \min_{x \in \Lambda_{N/2}} E [ | C(x) | ]  \le   \max_{x \in \Lambda_{N/2}} E [ | C(x) | ]\le  \max_{x \in \Lambda_N} E [ | C(x) | ] \le v_2 N^{2}$$ 
and then summing over $x$ in $\Lambda_N$ and in $\Lambda_{N/2}$, one gets the existence of $v_3, v_4$ such that for all large $N$,
$$  v_3 N^{d+2} \le E [ \sum_{n \le n_0} | C_n|^2 ] \le v_4  N^{d+2}. $$ 
 
Let us first show the following analogue of (4.10) in \cite {Aizenman}: 
\begin {proposition}
\label {firstprop}
For some fixed large $c_0$, with probability going to $1$ as $N \to \infty$, no loop-soup cluster (in $\Lambda_N$) contains more than $c_0 N^4 \log N$ points.
\end {proposition}

\begin {proof}
This is based on the fact that the Aizenman-Newman diagrammatic procedure \cite {AN} used in \cite {Aizenman} to bound the moments of $|C(x)|$ can be adapted to this loop-soup percolation setting. 
Let us first explain this in some detail
the case of the second moment. 
As mentioned above, one has
$$ E [ |C(x)|^2 ] = \sum_{y_1, y_2 \in \Lambda_N} P [ x \leftrightarrow y_1, x \leftrightarrow y_2 ].$$
When $x \leftrightarrow y_1, x \leftrightarrow y_2$ both occur, then it means that for some loop $\gamma$ in the cable-system loop-soup the events $\gamma \leftrightarrow x$, $\gamma \leftrightarrow y_1$ and $\gamma \leftrightarrow y_2$ occur disjointly (i.e., using disjoint sets of loops -- the loops may overlap, but each event is realized using different loops); we call ${\mathcal T}$ this event. [To see this, one can first choose a ``minimal'' chain of loops that join $x$ to $y_1$ (this means that one can not remove any  these loops from the chain without disconnecting $x$ to $y_1$) and then use a second ``minimal'' chain of loops that join $y_2$ to this first chain. The loop $\gamma$ will be the loop of the first chain that this second chain joins $y_2$ to.] 

In particular, it means that for at least one loop $\gamma$ in the cable-system loop-soup, one can find integer points $x_0$, $x_1$ and $x_2$ in $\Lambda_N$ that are at distance at most $1$ from $\gamma$ such that $x \leftrightarrow x_0$, $y_1 \leftrightarrow x_1$ and $y_2 \leftrightarrow x_2$ occur disjointly. We are going to treat differently the case where $\gamma$ visits at least two points of $\Z^d$ from the case where it visits less than two points. 

Let us introduce some notation and make some further preliminary comments: 
For each cable-system loop $\gamma$ that visits at least two integer points, one can look at its trace on $\Z^d$ that we denote by $l(\gamma)$, which is a discrete loop in $\Lambda_N$. Note that the collection ${L}$ of all $l(\gamma)$'s for $\gamma$ in the loop-soup ${\mathcal L}$ is a discrete random walk loop-soup in $\Lambda_N$, and that when an integer point is at distance at most $1$ from $\gamma$, it is also at distance at most $1$ from $l(\gamma)$.  If $|l| \ge 2$ denotes the number 
of steps of the discrete loop $l(\gamma)$, there are therefore at most $|l| \times (2d+1)$ possibilities for each of $x_0$, $x_1$ and $x_2$.

\begin{figure}[ht!]
\includegraphics[scale=.65]{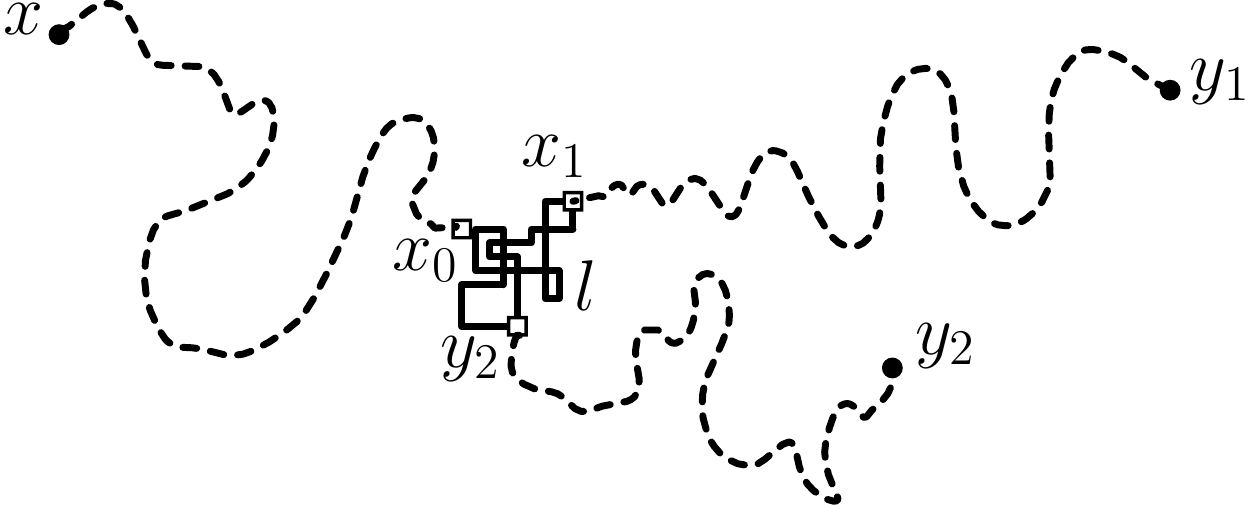} 
	\caption{Compared to the ``usual'' Aizenman-Newman tree expansion, one has also to sum over the loops that play the role of nodes of the tree, but this additional sum converges}
	\label{pic}
\end{figure}

For each given $x$, $y_1$ and $y_2$, we can now use the BK inequality to bound 
$P [ x \leftrightarrow y_1, x \leftrightarrow y_2 ]$ by the sum of the contributions described in (a) and (b) below: 

(a) The sum over all $x_0$, $x_1$ and $x_2$ that are all at distance at least $2$ from each other of the product 
$$P[ x_0 \leftrightarrow x ] P [ x_1 \leftrightarrow y_1 ] P [ x_2 \leftrightarrow y_2].$$ 
This sum corresponds to the contribution to  the event ${\mathcal T}$ of the cases where $\gamma$ visits at most one point of ${\Z}^d$. Note that for a given $x_0$, there are at most $(2d+1)^2$ choices (corresponding to the two steps or less needed to go from $x_0$ to $x_1$) for $x_1$ and $(2d+1)^2$ choices for $x_2$. 

(b) The sum over all discrete loops $l$ with $|l| \ge 2$ steps, of the sum over all $x_0$, $x_1$, $x_2$ that lie at distance at most $1$ of $l$, of the product 
$$ P[ l \in {L} ] P[ x_0 \leftrightarrow x ] P [ x_1 \leftrightarrow y_1 ] P [ x_2 \leftrightarrow y_2].$$
This sum corresponds to the case where the loop $\gamma$ in the event ${\mathcal T}$ visits at least two integer points (and we sum over all possible choices for $l(\gamma)$). 

Equation (\ref {twopoint}) shows the existence of a constant $w_0$ independent of $N$, such that for all $y, y' \in \Lambda_N$ (as it is easier to create a connection in $\Z^d$ than in $\Lambda_N$), 
$P [ y \leftrightarrow y'] \le w_0 /( 1 + \| y - y' \|^{d-2} )$;  it follows immediately (summing over all $y'$ that are in $y + \Lambda_{2N}$) that for some constant $w_1$, for all $N \ge 1$ and all $y \in \Lambda_N$,  
\begin {equation}
 \label {simplebound} 
 \sum_{y' \in \Lambda_N} P [ y \leftrightarrow y'] \le w_1 N^{2},
\end {equation} 
which is an inequality that we will now repeatedly use. 
For each choice of $x_0$, $x_1$ and $x_2$ (and possibly $l$ if we are in the case (a)), if we now sum over all choices of $y_1$ and $y_2$ in $\Lambda_N$, we can use (\ref {simplebound}) to see that 
\begin {eqnarray*}
 E [ |C(x)|^2 ]
&&\le
\sum_{x_0\in \Lambda_N}   P[ x_0 \leftrightarrow x ] (2d+1)^{4}  ( w_1 N^2)^2\\
&&
+ \sum_{(x_0, l)\in {\mathcal U}} \Bigl[ 
P [ x \leftrightarrow x_0] \times P[l \in {L} ] \times  (|l| (2d+1))^2 \times (w_1 N^2)^2  \Bigr]
\end {eqnarray*}
where ${\mathcal U}$ is the set of pairs $(x_0, l)$ satisfying (i)-(iii) where 
(i) $x_0 \in \Lambda_N$, (ii) the discrete loop $l$ has at least $2$ steps, and (iii) $x_0$ is at distance at most $1$ from $l$;   
the term $(2d+1)^{4}$ comes from the bound on the number of possible choices for $x_1$ and $x_2$ for a given $x_0$ in (a), and the term $(|l| (2d+1))^2$ comes from the possible choices for $x_1$ and $x_2$  in (b) for a given discrete loop $l$ with $|l| \ge 2$ steps). 

The first sum over $x_0$ is bounded $(2d+1)^4 w_1^3 N^6$ (using (\ref {simplebound}) again). For the second one, we can first note that for each given $x_0$, the expected number of discrete loops of length $m$ in a loop-soup in the whole of ${\Z}^d$ that pass through $x_0$ is given by the total mass of such loops under the discrete loop-measure, which is in turn expressed in terms of the probability that a random walk started from $x_0$ is back at $x_0$ after $m$ steps (see for instance \cite {WP} for such elementary considerations on loop-measures), which is bounded by some constant $w_2$ times $m^{-d/2}$. Hence, if we regroup the sum over all loops with the same length $m$, we see that the second sum over $(x_0,l)$  in ${\mathcal U}$ is bounded by  
$$ 
\sum_{x_0 \in \Lambda_N } \Bigl[ (2d+1) P [ x \leftrightarrow x_0] w_1^2 N^4\sum_{m \ge 2} 
[w_2 m^{-d/2} (m (2d+1))^2] \Bigr] .$$
The key point is now that when $ d/2 -2 > 1$, i.e., $d > 6$, then  $\sum_m m^{2 - d/2} < \infty$, so that finally, we see that this sum sum over $(x_0,l)$  in ${\mathcal U}$ is bounded by some constant times 
$$N^4 \sum_{x_0 \in \Lambda_N } 
P [ x \leftrightarrow x_0]$$
which in turn is also bounded by some constant times $N^6$ (using (\ref {simplebound}) again). Together with the bound for the sum in (a), we can therefore conclude that for some constant $w_3$, for all $N \ge 2$ and all $x \in \Lambda_N$, 
$$ E [ | C(x)|^2 ] \le w_3 N^6.$$ 
In summary, we see that $d>6$ is also the threshold at which the extended nature of the Brownian loops does not essentially influence the estimates compared to finite-range percolation. 

Similarly, for any $k  \ge 3$, by enumerating trees, and expanding in a similar way (this time, one has to sum over $k-1$ loops 
in the loop-soup that will be the nodes of the tree) using the Aizenman-Newman enumeration ideas, one obtains the existence of constants $w_4$ and $w_5$ such that for all $N$, $x$ and $k$, 
\begin {equation}
 \label {previous}
 E [|C(x)|^k ] \le w_4 k! w_5^k  N^{4k-2}.
 \end {equation}
If we then finally sum over all $x$ in $\Lambda_N$, we get that 
$$ E[ \sum_{n \le n_0} |C_n|^{k+1} ] = \sum_{x \in \Lambda_n} E [ |C(x)|^{k} ] \le w_4 k! w_5^k N^{d+4k-2}.$$ 
In particular, if $M$ denotes $\max |C_n|$, we get an upper bound for $E [ M^{k+1} ]$ 
from which one readily deduces the proposition by using Markov's inequality and choosing the appropriate $k$ (of the order of a constant times $\log N$). 
\end {proof} 

Let us now turn to the proliferation of large clusters: 
\begin {proposition} 
\label {secondprop}
With probability that tends to $1$ as $N$ tends to infinity, there exist more than $N^{d-6}/\log^2 N$ disjoint loop-soup clusters with diameter greater than $N/2$.  
\end {proposition}
The proof proceeds along the same lines as the analogous result (4.8) in \cite {Aizenman}: 
\begin {proof}
One can for instance define $B_1$ and $B_2$ to be the boxes obtained by shifting $\Lambda_{N/4}$ along the first-coordinate axis by $-N/2$ and $N/2$ respectively. Each of the two boxes has circa $(N/2)^d$ points in it, they at distance at least $N/4$ from $\partial \Lambda_N$, and they are at distance circa $N/2$ from each other.
Now, Corollary \ref {prop1} readily shows that if we define 
$$ X := \sum_n  |C_n \cap B_1| \times |C_n \cap B_2|,$$
then for some positive finite constant $b_1$,
$$ E [ X ] = E [ \sum_{x_1 \in B_1, x_2 \in B_2} 1_{x_1 \leftrightarrow x_2 }] 
\sim b_1 N^{d+2} $$ 
as $N  \to \infty$.  
On the other hand, one can bound the second moment 
$$ E [ X^2 ] = \sum_{x_1, y_1 \in B_1, x_2, y_2 \in B_2} P [ \E(x_1, x_2 ,y_1, y_2) ]$$
where $\E(x_1, x_2, y_1, y_2 ) := \{x_1 \leftrightarrow x_2, y_1 \leftrightarrow y_2 \}$
using the following remark (call the {\em truncation lemma} in \cite {Aizenman}): To check if  
$\E$ holds, one can first discover $C(x_1)$. If it does contain $x_2$, $y_1$ and $y_2$ (we call this event $\E_1$), then we know already that  $\E$ holds. The only other scenario (we call this event $\E_2=\E \setminus \E_1$) for which $\E$ holds is that $y_1 \in C(x_1)$, that neither $y_1$ nor $y_2$, are in $C(x_1)$, and then that for the remaining loop-soup percolation in the complement of $C(x_1)$ in the cable-graph, $y_1$ is connected to $y_2$. Clearly, 
$$P [ \E_2] = P[\E] - P [\E_1] \le  P [ x_1 \leftrightarrow y_1] P[x_2 \leftrightarrow y_2 ]$$
(the first probability in the product is an upper bound for the probability that $y_1 \in C(x_1)$ and that neither $y_1$ nor $y_2$ are in $C(x_1)$, and the second probability is an upper bound for the conditional probability that $x_2 \leftrightarrow y_2$ in the remaining domain). 
Summing this inequality over all $x_1. x_2, y_1, y_2$, and using (\ref {previous}) one immediately gets that
\begin {eqnarray*}
\lefteqn { E [ X^2 ] - E [ X ]^2 }
\\
&&= \sum_{ x_1,  y_1 \in B_1, x_2, y_2 \in B_2} [P [\E (x_1, x_2, y_1, y_2)] -  P [ x_1 \leftrightarrow y_1] P[x_2 \leftrightarrow y_2 ]]\\ 
&& \le 
\sum_{x_1, x_2, y_1, y_2 \in \Lambda_N}  P [\E_1 (x_1, x_2, y_1, y_2) ] 
\\
&&
= E [ \sum_{n \le n_0} |C_n|^4 ] \\
&& \le b_2 N^{d+10}\end {eqnarray*}
for some constant $b_2$ independent of $N$. 
Combining this bound of the variance of $X$ with the estimate  of its mean (and noting that $d+10 < 2 (d`+2)$ because $d>6$), we see that for all $\eps$, 
$$P [ X \in [(b_1-\eps) N^{d+2}, (b_1-\eps) N^{d+2}]] \to 1$$
as $N \to \infty$.
If $X$ denotes the number of clusters that intersect both $B_1$ and $B_2$, noting that with high probability, all quantities 
$|C_n \cap B_1|$ and $|C_n \cap B_2|$ are smaller than $c_0 N^4 \log N$ (because of Proposition \ref {firstprop}), we deduce that with a probability that goes to $1$ as $N \to \infty$, 
$$ X \ge \frac { (b_1/2) \times N^{d+2}} { (c_0  N^4 \log N)^2} = (b_1/2c_0^2) \times \frac {N^{d-6}}{ \log^2 N}.$$ 
\end {proof}

\subsection {Some final comments}

We conclude with the following comments (see \cite {Werner} for more in this direction): On the one hand, we have seen that when $N \to \infty$, there will typically 
be a large number of large clusters (say of diameter greater than $N/2$), but on the other hand, only a tight number of Brownian loops of diameter comparable to $N$. In fact, when $a \in (0,d)$, the $N^{a}$-th largest Brownian loop will have a diameter of the order of $N\times N^{-a/d +o(1)}$. This means for instance that an overwhelming fraction of the numerous large clusters will contain no loop of diameter greater than $N^b$ for $b > 6/d$. In other words, if we remove all loops of diameter greater than $N^b$, one will still have at least $N^{d-6+o(1)}$ large clusters, and the estimates for the two-point function will actually remain valid. If we fix $b \in (6/d, 1)$, since $N^b$ is also much smaller than the size $N$ of the box, we can interpret this cable-graph loop-soup percolation with cut-off as a critical (or near-critical) percolation model: If we scale everything down by a factor $N$: We 
are looking at a Poissonian family of small sets, and for the chosen parameters one observes macroscopic clusters (as $N \to \infty$). 

We will discuss further aspects of loop-soup cluster percolation and the structure of the GFF in high dimensions in \cite {Werner}. In particular, when $d \ge 9$, the relation with the integrated superbrownian excursions can be made more precise.  



\subsubsection*{Acknowledgement.} The author acknowledges the support of the Swiss National Science Foundation (SNF) grant  \#175505.

\bibliographystyle{plain}

\end{document}